\documentclass[12pt]{amsart}
\usepackage[top=1in, bottom=1in, left=1in, right=1in]{geometry}
\usepackage{amsfonts}
\usepackage{amsmath}
\usepackage{comment}
\usepackage{amssymb}
\usepackage{graphicx}
\usepackage{bbm}
\usepackage{comment}
\usepackage{mathrsfs}
\numberwithin{equation}{section}
\usepackage{times}

\usepackage{siunitx}
\usepackage[usenames,dvipsnames]{color}
\usepackage{comment}
\usepackage{xcolor}
\usepackage{mathtools}
\usepackage{bm}
\usepackage{esvect}
\usepackage{hyperref}

\hypersetup{
    colorlinks = true,
linkcolor={black},
urlcolor={blue},
citecolor={blue},    
urlcolor = {blue},
citebordercolor = {0.33 .58 0.33},
 linkbordercolor = {0.99 .28 0.23},
 breaklinks=true}
\newcommand{\Z}{\mathbb{Z}}

\newcommand{\m}{\mathrm{m}}

\newtheorem{thm}{Theorem}[section]

\newtheorem{conj}[thm]{Conjecture}

\theoremstyle{remark}

\usepackage{geometry}
\title{On a Conjecture about Sums Involving Farey Fractions}
\author{Anji Dong,  Xinyi Li, Vi Anh Nguyen}

\address{
Anji Dong: Department of Mathematics,
University of Illinois Urbana-Champaign,
Altgeld Hall, 1409 W. Green Street,
Urbana, IL, 61801, USA}
\email{anjid2@illinois.edu}

\address{
Xinyi Li: Department of Mathematics,
University of Illinois Urbana-Champaign,
Altgeld Hall, 1409 W. Green Street,
Urbana, IL, 61801, USA}
\email{xinyil25@illinois.edu} 

\address{
Vi Anh Nguyen: Department of Mathematics,
University of Illinois Urbana-Champaign,
Altgeld Hall, 1409 W. Green Street,
Urbana, IL, 61801, USA}
\email{vianhan2@illinois.edu}

\begin{document}
\setcounter{tocdepth}{1}
\begin{abstract}
   In this paper, we prove a conjecture by Daniele Mundici on the sum of squared distances between consecutive elements in the $Q$-th Farey sequence for $Q\in\Z$ and $Q\geq 2$. 
\end{abstract}
\maketitle

\section{Introduction}
Let $Q\geq 2$ be an integer. The $Q$-th Farey sequence $F_Q$ is defined as follows:
\begin{align*}
    F_Q :=\left\{\frac{a}{q}: 1\leq a\leq q\leq Q,\ (a,q)=1\right\}.
\end{align*}

Assume $|F_Q|=N_Q$ such that
\begin{align*}
   F_Q :=\{\gamma_1,\gamma_2,\dots,\gamma_{N_Q}\} 
\end{align*}
with $\gamma_1<\gamma_2<\dots<\gamma_{N_Q}$. For simplicity, we will write $N=N_Q$, but keeping the dependence of $Q$ in mind. 

Farey sequences have been studied for a long time, with one of the more extensively studied topics being the distribution of spacings between consecutive Farey fractions in a subinterval $I$ of $(0,1]$. One direction is to explore the distribution of powers of the consecutive spacing. For example, the case $I=(0,1]$ has been considered in \cite{hall} for any power $m\geq 2$. Boca, Cobeli, and Zaharescu later extended the results to subintervals of $(0,1]$ in \cite{zaharescu}. Other results along this direction can be found in \cite{kanemitsu2,kanemitsu1,  moschevitin}. Another possible direction is to study the $h$-th level consecutive spacing for any $h\geq 2$, for example, see \cite{boca,zaharescu}. One can also find results in other directions in \cite{Huxley, kargaev}. 

In this paper, we go along the first direction and study the distribution when the power of consecutive spacing is $2$. Let 
\begin{align}
    S_2(Q) := \sum_{j=1}^{N-1} (\gamma_{j+1}-\gamma_j)^2,\label{def:S2(Q)}
\end{align}
which is the sum of squared distances between consecutive elements in the $Q$-th Farey sequence.  In \cite{moschevitin}, it was mentioned that this quantity is an important special case in the study of sums of powers of consecutive distances between Farey points because it can be interpreted as the average length of the interval $[\gamma_{j}, \gamma_{j+1})$ which a random uniformly distributed point in $(0,1]$ falls in. In \cite{hall}, Hall proved that
\begin{align}
    S_2(Q) = \frac{12}{\pi^2 Q^2}\left(\log Q + \gamma + 1/2 - \frac{\zeta'(2)}{\zeta(2)}\right) + O\left(\frac{\log^2 Q}{Q^3}\right).\label{eq: Hall's theorem}
\end{align}
Later, in \cite{kanemitsu2,kanemitsu1}, Kanemitsu et al. refined Hall's result by showing that the error term can be reduced to $O(\log Q/Q^3)$. The established asymptotic formula is as follows:
\begin{align}
 S_2(Q) = \frac{12}{\pi^2 Q^2}\left(\log Q + \gamma + 1/2 - \frac{\zeta'(2)}{\zeta(2)}\right) +\frac{4U(Q)\log Q}{Q^3}+ O\left(\frac{\log Q}{Q^3}\right),\label{eq:Kanemitsu's theorem}
\end{align}
where 
\[
U(x) = \sum_{Q\leq x}\frac{\mu(Q)}{Q}\left(\frac{x}{Q}-\left[\frac{x}{Q}\right]-\frac{1}{2}\right).
\]

Define
\begin{align}
    C(Q) := \frac{S_2(Q)\cdot Q^2}{\log Q}. \label{def: C(n)}
\end{align}
Upon checking the first $1000$ values of $Q$, Daniele Mundici came up with the following conjecture:
\begin{conj}\label{conj: main conjecture}
    \[
    C(Q)<3 \text{ for all } Q>1,
    \]
     where $C(Q)$ is defined in \eqref{def: C(n)}.
\end{conj}

In follow-up private conversations, Mundici suggested that Conjecture \ref{conj: main conjecture} might be provable assuming the Riemann Hypothesis. In this paper, we show that the assumption of the Riemann Hypothesis is not needed, and prove Mundici's original conjecture unconditionally. Our main result is the following.

\begin{thm}\label{thm: main theorem}
    \[
    C(Q)<3 \text{ for all } Q>1,
    \]
    where $C(Q)$ is defined in \eqref{def: C(n)}.
\end{thm}
Along the result, we give an asymptotic formula of $S_2(Q)$ with an explicit error bound. 
\begin{thm}\label{thm: Theorem 2}
    \begin{align*}
    S_2(Q) &= \frac{12\log Q}{\pi^2 Q^2}-\frac{2}{Q^2}\frac{\zeta'(2)}{\zeta(2)^2} +(2\gamma+1)\frac{6}{Q^2\pi^2}+R_{14},
\end{align*}
where 
\[
R_{14}\leq \frac{64(\log Q)^2+106\log Q+269}{Q^3}.
\]
\end{thm}
 The proof we give here is different from Hall's proof in \cite{hall}.
\section{Proof of Theorem \ref{thm: Theorem 2}}

In this section, we outline the proof of Theorem \ref{thm: Theorem 2}.

\begin{proof}
Recall some basic properties of Farey sequence: any two consecutive element $a/q$ and $a'/q'$ must satisfy $a/q-a'/q' = 1/qq'$, $q+q'>Q$, and $(q,q')=1$ assuming $a/q$ precedes $a'/q'$. Therefore, $S_2(Q)$ can be rewritten as
\begin{align}
     S_2(Q) &= \sum_{\substack{1\leq q,q'\leq Q\\q+q'>Q\\(q,q')=1} }\frac{1}{(qq')^2}\notag\\
     &=\sum_{\substack{(q,q')\in\Omega\\(q,q')=1}} \frac{1}{q^2q'^2},
\end{align}
where $\Omega$ is the region in the plane which is enclosed by the lines $x=1,\ y =1, \ y=-x+Q, x=Q, $ and $y=Q$. Note that $\#\{(q,q')\in\Omega\}$ counts the integer lattice points in $\Omega$.

Upon detecting the coprimality condition using the M\"obius function, we may further rewrite $S_2(Q)$ as
\begin{align*}
    S_2(Q) &= \sum_{\substack{(q,q')\in\Omega}} \frac{1}{q^2q'^2}\sum_{\substack{d\mid q\\d\mid q'}}\mu(d)\\
    &=\sum_{1\leq d\leq Q}\frac{\mu(d)}{d^4}\sum_{\substack{(dm,dn)\in\Omega\\m,n\leq Q/d}}\frac{1}{m^2n^2}\\
    &=\sum_{1\leq d\leq Q}\frac{\mu(d)}{d^4}\sum_{\substack{(m,n)\in\Omega/d}}\frac{1}{m^2n^2},
\end{align*}
where $\Omega/d$ denotes the region in the plane which is enclosed by the lines $x=1,\ y =1, \ y=-x+Q/d, x=Q/d, $ and $y=Q/d$. In other words,
\begin{align}
    S_2(Q) &= \sum_{1\leq d\leq Q}\frac{\mu(d)}{d^4}\sum_{\substack{1\leq m,n\leq \lfloor Q/d\rfloor\\m+n\geq \lfloor Q/d\rfloor +1}}\frac{1}{m^2n^2}.
  \label{eq: S2(Q) intermediate step}
\end{align}
Denote $\lfloor Q/d\rfloor$ by $k$ for simplicity and the inner sum by $A(k)$. Then,  
\begin{align}
   A(k) =  \left(\sum_{1\leq m\leq k}\frac{1}{m^2}\right)^2-\sum_{\substack{1\leq m,n\leq k\\m+n\leq k }}\frac{1}{m^2n^2} = B_k^2-\sum_{s=2}^k\sum_{m=1}^{s-1}\frac{1}{m^2(s-m)^2},\label{eq:another form}
\end{align}
where
\[
B_n := \sum_{1\leq m\leq n}\frac{1}{m^2},
\]
and the second equality follows from writing $s=m+n$. We focus on the second summand in \eqref{eq:another form}. Observe that 
\begin{align*}
   \sum_{s=2}^k\sum_{m=1}^{s-1}\frac{1}{m^2(s-m)^2} &=\sum_{s=2}^k\sum_{m=1}^{s-1}\frac{1}{s^2}\left( \frac{1}{m^2} + \frac{1}{(s-m)^2} \right) + \frac{2}{s^3}\left( \frac{1}{m} + \frac{1}{s-m} \right) \\
   &=\sum_{s=2}^k\frac{2B_{s-1}}{s^2}+\frac{4H_{s-1}}{s^3},
\end{align*}
where 
$H_{n}$ is the $n$-th harmonic number, which is defined as
\[
H_n := \sum_{m=1}^n \frac{1}{m}.
\]
Therefore, 
\begin{align*}
    A(k) = B_k^2 - 2\sum_{s=2}^k\frac{B_{s-1}}{s^2}-4\sum_{s=2}^k\frac{H_{s-1}}{s^3}.
\end{align*}
Now the second summand can be rewritten as 
\begin{align*}
   2\sum_{s=2}^k\frac{B_{s-1}}{s^2} &= 2\sum_{s=2}^k\sum_{m=1}^{s-1}\frac{1}{m^2s^2} = 2\sum_{1\leq m<s\leq k}\frac{1}{m^2s^2}\\
   &=\sum_{1\leq m\neq s\leq k}\frac{1}{m^2s^2} = B_k^2-F_k,
\end{align*}
where 
\[
F_k = \sum_{m=1}^k\frac{1}{m^4}.
\]
Substituting back, we obtain
\begin{align*}
    A(k) = F_k-4\sum_{s=2}^k\frac{H_{s-1}}{s^3}.
\end{align*}
Note that if $k$ goes to infinity, then the original definition of $A(k)$ implies that $A(k)$ goes to $0$, and $F_k$ goes to $\zeta(4)$. Therefore, we may write
\begin{align*}
    4\sum_{s=2}^k\frac{H_{s-1}}{s^3} = \zeta(4)-4\sum_{s\geq k+1}\frac{H_{s-1}}{s^3}. 
\end{align*}
Substituting this back, we obtain
\begin{align*}
    A(k) &= F_k - \zeta(4)+4\sum_{s\geq k+1}\frac{H_{s-1}}{s^3}\\
    &=4\sum_{s\geq k+1}\frac{H_{s-1}}{s^3}-\sum_{s\geq k+1}\frac{1}{s^4}.
\end{align*}

By \cite[p. 114]{knuth}, for $n\geq 1$, 
\[
H_n = \log n +\gamma +\frac{1}{2n}+R_1,
\]
where $|R_1|<\frac{1}{10n^2}$. Therefore,
\begin{align*}
    A(k) = \sum_{s\geq k+1}\frac{4(\log s+\gamma )}{s^3}+R_2,
\end{align*}
where $$|R_2|\leq \sum_{s\geq k}\frac{3}{s^4}+\frac{2}{5s^5}\leq \frac{11}{10k^3}.$$
The main term is a decreasing function that satisfies 
\[
\int_{k+1}^\infty \frac{4(\log t+\gamma )}{t^3}dt\leq \sum_{s\geq k+1}\frac{4(\log s+\gamma )}{s^3}\leq \int_{k}^\infty \frac{4(\log t+\gamma )}{t^3}dt.
\]
Therefore, 
\begin{align*}
    A(k) = \int_{k}^\infty \frac{4(\log t+\gamma )}{t^3}dt+R_3,
\end{align*}
where
\begin{align*}
    |R_3|\leq \int_{k}^{k+1}\frac{4(\log t+\gamma )}{t^3}dt+\frac{11}{10k^3}.
\end{align*}
Using integration by parts, we finally obtain that
\begin{align*}
    A(k) =  \frac{2\log k+2\gamma +1}{k^2}+R_4,
\end{align*}
where 
\begin{align*}
    |R_4|\leq \frac{4(\log(k+1) + \gamma)+11/10}{k^3}\leq \frac{4\log k+8}{k^3}.
\end{align*}
Substituting it back in \eqref{eq: S2(Q) intermediate step}, 
\begin{align}
     S_2(Q) &= 2\sum_{1\leq d\leq Q}\frac{\mu(d)}{d^4}\cdot\frac{\log k}{k^2} +(2\gamma+1)\sum_{1\leq d\leq Q}\frac{\mu(d)}{d^4} \cdot\frac{1}{k^2}+R_5\label{eq: new S2(Q) form}
\end{align}
where 
\begin{align*}
    |R_5|&\leq 4\sum_{1\leq d\leq Q}\frac{1}{d^4}\cdot\frac{\log k}{k^3}+8\sum_{1\leq d\leq Q}\frac{1}{d^4}\cdot\frac{1}{k^3}\\
    &\leq \frac{32}{Q^3}\sum_{1\leq d\leq Q}\frac{\log(Q/d)}{d}+\frac{64}{Q^3}\sum_{1\leq d\leq Q}\frac{1}{d}\\
    &\leq \frac{32(\log Q)^2+64\log Q}{Q^3},
\end{align*}
and the second inequality follows from the fact that $\lfloor x\rfloor \geq x/2$. Recall that $k=\lfloor Q/d\rfloor$. We have \begin{align*}
    \sum_{1\leq d\leq Q}\frac{\mu(d)}{d^4}\frac{1}{k^2} &= \sum_{1\leq d\leq Q}\frac{\mu(d)}{d^4}\frac{1}{ (Q/d)^2}+\sum_{1\leq d\leq Q}\frac{\mu(d)}{d^4}\left(\frac{1}{\lfloor Q/d\rfloor^2}-\frac{1}{ (Q/d)^2}\right).
\end{align*}
Using $\sum_{1\leq d\leq \infty}\dfrac{\mu(d)}{d^2}=\dfrac{1}{\zeta(2)}=\dfrac{6}{\pi^2}$, we obtain
\begin{align*}
    \sum_{1\leq d\leq Q}\frac{\mu(d)}{d^4}\frac{1}{k^2} &= \frac{6}{Q^2\pi^2}+\sum_{1\leq d\leq Q}\frac{\mu(d)}{d^4}\left(\frac{1}{\lfloor Q/d\rfloor^2}-\frac{1}{ (Q/d)^2}\right)+R_6.
\end{align*}
where $|R_6|\leq 1/Q^3$. Now split the range of $d$ into $d\leq Q/2$ and $d>Q/2$. Then
\begin{align*}
    \Big|\sum_{1\leq d\leq Q/2}\frac{\mu(d)}{d^4}\left(\frac{1}{k^2}-\frac{1}{ (Q/d)^2}\right)\Big| &\leq  2\sum_{1\leq d\leq Q/2}\frac{1}{d^4}\frac{1}{\lfloor Q/d\rfloor^2Q/d}\\
    &\leq 2 \sum_{1\leq d\leq Q/2}\frac{1}{d^4}\frac{1}{(Q/2d)^2\cdot Q/d}\\
    &\leq 8Q^{-3}\sum_{1\leq d\leq Q/2}\frac{1}{d}\leq 8\frac{\log Q}{Q^3}.
\end{align*}
 Now when $d>Q/2$, $\lfloor Q/d\rfloor=1$, so 
\begin{align*}
    \Big|\sum_{Q/2<d\leq Q}\frac{\mu(d)}{d^4}\left(\frac{1}{\lfloor Q/d\rfloor^2}-\frac{1}{ (Q/d)^2}\right)\Big| &\leq \sum_{Q/2<d\leq Q}\frac{8}{d^4}\leq \frac{64}{Q^3},
\end{align*}
 by observing that $1/d^4\leq 1/(Q/2)^4$. Therefore, 
\begin{align}
   \sum_{1\leq d\leq Q}\frac{\mu(d)}{d^4k^2}=  \sum_{1\leq d\leq Q}\frac{\mu(d)}{d^4}\frac{1}{\lfloor Q/d\rfloor^2}  = \frac{6}{Q^2\pi^2}+R_8,\label{eq:no log contribution}
\end{align}
where $|R_8|\leq (64+8\log Q)Q^{-3}$.  Applying the Mean Value Theorem, we have
\[
   k^{-2} =((Q/d)^{-2}+\delta_1)
\]
with
$|\delta_1|\leq \frac{16}{(Q/d)^3}$ and
\[
|\log k-\log (Q/d)| =\log\frac{Q/d}{k}\leq \log \frac{Q/d}{Q/d-1}\leq 4(Q/d)^{-1}
\]
for $d\leq Q/2$ and $|\log k-\log (Q/d)| =\log (Q/d)<4(Q/d)^{-1}$ for $d>Q/2$. Therefore, 
\begin{align*}
    \sum_{1\leq d\leq Q}\frac{\mu(d)\log k}{d^4k^2} &= \sum_{1\leq d\leq Q}\frac{\mu(d)}{d^4}\left((\log(Q/d)+\delta_2)((Q/d)^{-2}+\delta_1)\right)\\
    &= \sum_{1\leq d\leq Q}\frac{\mu(d)}{d^4}\left(\frac{d^2\log(Q/d)}{Q^2}+R_{9}\right)
\end{align*}
where $|\delta_2|\leq 4(Q/d)^{-1}$  and $|R_{9}|\leq \frac{16d^3\log Q}{Q^3}+\frac{4d^3}{Q^3}+\frac{64d^4}{Q^4}$. Upon simplification and using the fact that $|\sum_{d\leq N}\mu(d)/d|\leq 1$ for any $N\in\mathbb{N}$, we get
\begin{align*}
    \sum_{1\leq d\leq Q}\frac{\mu(d)\log k}{d^4k^2}&= \frac{\log Q}{Q^2}\sum_{1\leq d\leq Q/2}\frac{\mu(d)}{d^2}-\frac{1}{Q^2}\sum_{1\leq d\leq Q/2}\frac{\mu(d)\log d}{d^2}+R_{10}\\
    &=\frac{\log Q}{Q^2}\left(\frac{6}{\pi^2}+R_{11}\right)-\frac{1}{Q^2}\left(\frac{\zeta'(2)}{\zeta(2)^2}+R_{12}\right)+R_{10},
\end{align*}
where $|R_{10}|\leq \frac{16(\log Q)^2+4\log Q+64}{Q^3}$, $|R_{11}|\leq 4/Q$, and $|R_{12}|\leq \frac{4\log Q}{Q}$. Thus, \begin{align}
     \sum_{1\leq d\leq Q}\frac{\mu(d)\log k}{d^4k^2}&=\frac{6\log Q}{\pi^2 Q^2}-\frac{1}{Q^2}\frac{\zeta'(2)}{\zeta(2)^2}+R_{13},\label{eq:log power contribution}
\end{align}
where $|R_{13}|\leq\frac{16(\log Q)^2+12\log Q+64}{Q^3}$. Substituting \eqref{eq:no log contribution} and \eqref{eq:log power contribution} back in \eqref{eq: new S2(Q) form}, and estimating $R_5$ similarly, we finally arrive at
\begin{align*}
    S_2(Q) &= \frac{12\log Q}{\pi^2 Q^2}-\frac{2}{Q^2}\frac{\zeta'(2)}{\zeta(2)^2} +(2\gamma+1)\frac{6}{Q^2\pi^2}+R_{14},
\end{align*}
where 
\[
|R_{14}|\leq \frac{64(\log Q)^2+106\log Q+269}{Q^3}.
\]
This finishes the proof of Theorem \ref{thm: Theorem 2}.  
\end{proof}

\section{Proof of Theorem \ref{thm: main theorem}}

Now we are ready to prove Conjecture \ref{conj: main conjecture}. By Theorem \ref{thm: Theorem 2}, we have
\[
C(Q) = \frac{12}{\pi^2}+\frac{-2\frac{\zeta'(2)}{\zeta(2)^2}+(2\gamma+1)\frac{6}{\pi^2}}{\log Q}+\frac{Q^2R_{14}(Q)}{\log Q}, 
\]
hence,
\begin{align}
C(Q) \leq \frac{12}{\pi^2}+\frac{-2\frac{\zeta'(2)}{\zeta(2)^2}+(2\gamma+1)\frac{6}{\pi^2}}{\log Q}+\frac{64\log Q}{Q}+\frac{106}{Q}+\frac{269}{Q\log Q}.\label{eq: C(n) upper bound}
\end{align}
Denote the function in terms of $Q$ on the right side of \eqref{eq: C(n) upper bound} by $G(Q)$. Then, observe that 
\[
G'(x) =-\frac{-2\frac{\zeta'(2)}{\zeta(2)^2}+(2\gamma+1)\frac{6}{\pi^2}}{x(\log x)^2}+\frac{64(1-\log x)}{x^2}-\frac{106}{x^2}-\frac{269(\log x+1)}{x^2(\log x)^2}<0
\]
for $x\geq 2$. Therefore, $G(x)$ is strictly decreasing for $x\geq 2$. A direct computation shows that 
\[
G(374)<3. 
\]
Thus, it only remains to check for integers $2\leq Q\leq 373$. Numerical computations confirmed that the maximum value of $C(Q)$ for $2\leq Q\leq 500$ is approximately $2.885390081777927$, which occurs at $Q=2$. This finishes the proof of Theorem \ref{thm: main theorem}. 

\section{Some Open Problems}
We conclude the paper with some open questions for interested readers. 

Let $I\subseteq(0,1]$ be a fixed interval and let $\gamma_1<\gamma_2<\cdots<\gamma_{N_I(Q)}$ denote the Farey fractions of order $Q$ from $I$. Then, one may generalize the definition of $S_2(Q)$ in \eqref{def:S2(Q)} as \begin{align*}
    S_{2,h, I}(Q) = \sum_{j=1}^{N_I(Q)} (\gamma_{j+h}-\gamma_j)^2. 
\end{align*}  
Note that this generalization goes along the second direction mentioned in the introduction. When $h=1$, it was proved in \cite{zaharescu} that 
\[
S_{2,1,I}(Q)= |I| S_2(Q)+2c_IQ^{-2}+O_\varepsilon(Q^{-21/10+\varepsilon}),
\]
where $c_I$ is some real constant depending on $I$ and vanishes when $I = (0,1]$. It's natural to ask whether Mundici's conjecture admits an analogous formulation in this localized setting. In other words, define 
\[
C(Q,I) = \frac{S_{2,1,I}(Q)Q^2}{|I|\log Q}.
\]
Does there exist a result analogous to Theorem \ref{thm: main theorem}, stating that
\[
C(Q,I)<K_I
\]
for some \textbf{effectively computable constants} $K_I$ and $Q_I$ that possibly depend on $I$, and for all $Q > Q_I$? Moreover, can one take $K_I$ to be an absolute constant independent of $I$? A more general open problem would be to generalize such results to $S_{2,h,I}(Q)$ for any $h\geq 1$.

\section*{Funding}
A.D. is supported by the Shaff-Andrews Fellowship, Department of Mathematics, University of Illinois Urbana-Champaign.
\section*{Acknowledgements}
The authors are grateful to Daniele Mundici for raising the question that led to this work.

\end{document}